\documentclass{amsart}%
\usepackage{amsfonts}
\usepackage{amsmath}
\usepackage{amssymb}
\usepackage{graphicx}%
\providecommand{\U}[1]{\protect\rule{.1in}{.1in}}
\newtheorem{theorem}{Theorem}

\newtheorem{corollary}[theorem]{Corollary}

\newtheorem{example}[theorem]{Example}

\newtheorem{lemma}[theorem]{Lemma}

\newtheorem{remark}[theorem]{Remark}

\begin{document}
\title{Non-polynomial entire solutions to $\sigma_{k}$ equations }
\author{Micah Warren}
\address{Fenton Hall\\
University of Oregon\\
Eugene, OR 97403}
\email{micahw@uoregon.edu}
\thanks{The author's work is supported in part by the NSF via DMS-1161498}
\maketitle

\begin{abstract}
For $2k=n+1$, we exhibit non-polynomial solutions to the Hessian equation%
\[
\sigma_{k}(D^{2}u)=1
\]
on all of $%
\mathbb{R}
^{n}.$

\end{abstract}

\section{Introduction}

In this note, we demonstrate the following.\ 

\begin{theorem}
For
\[
n\geq2k-1,
\]
there exist non-polynomial elliptic entire solutions to the equation
\begin{equation}
\sigma_{k}\left(  D^{2}u\right)  =1\label{main}%
\end{equation}
on $%
\mathbb{R}
^{n}.$ 
\end{theorem}

\begin{corollary}
For all $n\geq3,$ there exist on $%
\mathbb{R}
^{n}$ non-polynomial entire solutions to
\begin{equation}
\sigma_{2}\left(  D^{2}u\right)  =1.\label{sig2}%
\end{equation}

\end{corollary}

\ For $k=1,$ the entire harmonic functions in the plane arising as real parts
of analytic functions are classically known. $\ $For $k=n,$ the famous
Bernstein result of J\"{o}rgens \cite{J54}, Calabi \cite{MR0106487}, and
Pogorelov \cite{Pog72} states that all entire solutions to the
Monge-Amp\`{e}re equation are quadratic. \ \ \ Chang and Yuan \cite{CY2010}
have shown that any entire convex solution to (\ref{sig2}) in any dimension
must be quadratic. \ To the best of our knowledge, for $1<k<n$, the examples
presented here are the first known non-trivial entire solutions to $\sigma
_{k}$ equations. \

The special Lagrangian equation is the following%
\begin{equation}
\sum_{i=1}^{n}\arctan\lambda_{i}=\theta\label{sLag}%
\end{equation}
(here $\lambda_{i}$ are eigenvalues of $D^{2}u)$ for
\[
\theta\in\left(  -\frac{n}{2}\pi,\frac{n}{2}\pi\right)
\]
a constant. \ Fu \cite{Fu} showed that when $n=2$ and $\theta\neq0$ all
solutions are quadratic. When $n=2$ and $\theta=0$ the equation (\ref{sLag})
becomes simply the Laplace equation, which admits well-known non-polynomial
solutions. \ Yuan \cite{YuanInventiones} showed that all convex solutions to
special Lagrangian equations are quadratic. \ 

The critical phase for special Lagrangian equations is
\[
\theta=\frac{n-2}{2}\pi.
\]
Yuan \cite{YY06} has shown that for values above the critical phase, all
entire solutions are quadratic. \ On the other hand, by adding a quadratic to
a harmonic function, one can construct nontrivial entire solutions for phases
\newline%
\[
\left\vert \theta\right\vert <\frac{n-2}{2}\pi.
\]
By \cite{HL} when $n=3$, the critical equation
\[
\sum_{i=1}^{3}\arctan\lambda_{i}=\frac{\pi}{2}%
\]
is equivalent to the equation (\ref{sig2}). Thus Corollary 2 answers the
critical phase Bernstein question when $n=3.$ \ In the process, we also show
the following. \ 

\begin{theorem}
There exists a special Lagrangian graph in $%
\mathbb{C}
^{3}$ over $%
\mathbb{R}
^{3}$ that does not graphically split. \ 
\end{theorem}

Harvey and Lawson \cite{HL}, show that a\ graph
\[
(x,\nabla u(x))\subset%
\mathbb{C}
^{n}%
\]
is special Lagrangian and a minimizing surface if and only if $u$ satisfies
(\ref{sLag}). \ We say a graph splits graphically when the function $u$ can be
written the sum of two functions in independent variables. \ 

There are still many holes in the Bernstein picture for $\sigma_{k}$
equations. To begin with, when $n=4$ the existence of interesting solutions to
$\sigma_{3}=1.$ For special Lagrangian equations the existence of critical
phase solutions when $n\geq4$ is open. \ 

\section{\bigskip Proof}

We will assume that $n$ is odd and
\[
2k=n+1.
\]
We construct a solution $u$ on $%
\mathbb{R}
^{n}.$ \ The general result will follow by noting that if we define
\[
\tilde{u}:%
\mathbb{R}
^{n}\times%
\mathbb{R}
^{m}\rightarrow%
\mathbb{R}
\]
via%
\[
\tilde{u}(z,w)=u(z)
\]
then
\[
\sigma_{k}\left(  D^{2}\tilde{u}\right)  =\sigma_{k}\left(  D^{2}u\right)  =1.
\]

Consider functions on $%
\mathbb{R}
^{n-1}\times%
\mathbb{R}
$ of the form%
\[
u(x,t)=r^{2}e^{t}+h(t)
\]
where
\[
r=\left(  x_{1}^{2}+x_{2}^{2}+...+x_{n-1}^{2}\right)  ^{1/2}.
\]
Compute the Hessian, rotating $%
\mathbb{R}
^{n-1}$ so that $x_{1}=r:$%
\begin{align}
D^{2}u &  =\left(
\begin{array}
[c]{ccccc}%
2e^{t} & 0 & ... & 0 & 2re^{t}\\
0 & 2e^{t} & 0 & ... & 0\\
... & 0 & ... & 0 & ...\\
0 & ... & 0 & 2e^{t} & 0\\
2re^{t} & 0 & ... & 0 & r^{2}e^{t}+h^{\prime\prime}(t)
\end{array}
\right)  \nonumber\\
&  =e^{t}\left(
\begin{array}
[c]{ccccc}%
2 & 0 & ... & 0 & 2r\\
0 & 2 & 0 & ... & 0\\
... & 0 & ... & 0 & ...\\
0 & ... & 0 & 2 & 0\\
2r & 0 & ... & 0 & r^{2}+e^{-t}h^{\prime\prime}(t)
\end{array}
\right)  .\label{elliptic}%
\end{align}

\bigskip We then compute. \ The $k$-th symmetric polyomial is given by the sum
of $k$-minors. \ Let
\begin{equation}
S=\left\{  \alpha\subset\left\{  1,...,n\right\}  :\left\vert \alpha
\right\vert =k\right\}  ,
\end{equation}
and let%
\begin{align*}
A &  =\left\{  \alpha\in S:1\in\alpha\right\}  \\
B &  =\left\{  \alpha\in S:n\in\alpha\right\}  .
\end{align*}
We express $S$ as a disjoint union
\[
S=\left(  A\cap B\right)  \cup(B\backslash A)\cup\left(  S\backslash B\right)
.
\]
Define
\[
\sigma_{k}^{(\alpha)}=\det\left(
\begin{array}
[c]{c}%
k\times k\text{ matrix with}\\
\text{ row and columns}\\
\text{chosen from }\alpha\text{ }%
\end{array}
\right)  .
\]
\ \ For $\alpha\in\left(  A\cap B\right)  $ we have
\[
\sigma_{k}^{(\alpha)}=\det\left(  e^{t}\left(
\begin{array}
[c]{ccccc}%
2 & 0 & ... & 0 & 2r\\
0 & 2 & 0 & ... & 0\\
... & 0 & ... & 0 & ...\\
0 & ... & 0 & 2 & 0\\
2r & 0 & ... & 0 & r^{2}+e^{-t}h^{\prime\prime}%
\end{array}
\right)  \right)  ,
\]
that is%
\[
\sigma_{k}^{(\alpha)}=e^{kt}2^{k-2}\left(  2r^{2}+2e^{-t}h^{\prime\prime
}-4r^{2}\right)  .
\]
Next, for $\alpha\in B\backslash A$,%
\[
\sigma_{k}^{(\alpha)}=\det\left(  e^{t}\left(
\begin{array}
[c]{cccc}%
2 & 0 & ... & 0\\
0 & ... & 0 & 0\\
... & 0 & 2 & ...\\
0 & 0 & ... & r^{2}+e^{-t}h^{\prime\prime}%
\end{array}
\right)  \right)  ,
\]
that is
\[
\sigma_{k}^{(\alpha)}=e^{kt}2^{k-1}\left(  r^{2}+e^{-t}h^{\prime\prime
}\right)  .
\]
Finally, for $\alpha\in\left(  S\backslash B\right)  $ we have
\[
\sigma_{k}^{(\alpha)}=\det\left(  e^{t}\left(
\begin{array}
[c]{ccc}%
2 & 0 & ...\\
0 & ... & 0\\
... & 0 & 2
\end{array}
\right)  \right)  ,
\]
that is
\[
\sigma_{k}^{(\alpha)}=e^{kt}2^{k}.
\]
We sum these up:%
\[
\sigma_{k}\left(  D^{2}u\right)  =\sum_{\alpha\in\left(  A\cap B\right)
}\sigma_{k}^{\alpha}+\sum_{\alpha\in(B\backslash A}\sigma_{k}^{\alpha}%
+\sum_{\alpha\in\left(  S\backslash B\ \right)  }\sigma_{k}^{\alpha}.
\]
Counting, we get%
\begin{align}
\sigma_{k}\left(  D^{2}u\right)   &  =\binom{n-2}{k-2}e^{kt}2^{k-1}\left(
e^{-t}h^{\prime\prime}-r^{2}\right)  \label{e1}\\
&  +\binom{n-2}{k-1}e^{kt}2^{k-1}\left(  r^{2}+e^{-t}h^{\prime\prime}\right)
\nonumber\\
&  +\binom{n-1}{k}e^{kt}2^{k}.\nonumber
\end{align}
Grouping the terms, we see%
\begin{align*}
\sigma_{k}\left(  D^{2}u\right)   &  =e^{kt}2^{k-1}\left[  -\binom{n-2}%
{k-2}+\binom{n-2}{k-1}\right]  r^{2}\\
&  +e^{kt}2^{k-1}\left[  \binom{n-2}{k-2}+\binom{n-2}{k-1}\right]
e^{-t}h^{\prime\prime}\\
&  +e^{kt}2^{k-1}2\binom{n-1}{k}.
\end{align*}
Now%
\[
-\binom{n-2}{k-2}+\binom{n-2}{k-1}=-\frac{\left(  n-2\right)  !}{\left(
n-k\right)  !(k-2)!}+\frac{\left(  n-2\right)  !}{\left(  n-k-1\right)
!(k-1)!}.
\]
This vanishes if and only if
\[
1=\frac{\left(  n-k\right)  !(k-2)!}{\left(  n-k-1\right)  !(k-1)!}%
=\frac{\left(  n-k\right)  }{(k-1)},
\]
or precisely when
\[
n-k=k-1
\]
or
\[
2k=n+1.
\]
Thus for this choice of $k$, (\ref{e1}) becomes%
\[
\sigma_{k}\left(  D^{2}u\right)  =A_{n,k}e^{\left(  k-1\right)  t}%
h^{\prime\prime}+B_{n,k}e^{kt}%
\]
for some constants $A_{n,k},B_{n,k}.$ \ Setting to this expression to $1,$ we
solve for $h^{\prime\prime}(t)$%
\begin{equation}
h^{\prime\prime}(t)=\frac{1-B_{n,k}e^{kt}}{A_{n,k}e^{\left(  k-1\right)  t}%
},\label{hequation}%
\end{equation}
noting the right-hand side is a smooth function in $t.$ \ \ Integrating twice
in $t\,\ $yields solutions to (\ref{hequation}) and hence to (\ref{main}). \ 

To see that the equation is elliptic, we first note that inspecting
(\ref{elliptic}) the $n-2$ eigenvalues in the middle must be positive. \ Of
the remaining two, at least one must be positive as the diagonal (of the
$2\times2$ matrix) contains at least one positive entry. \ \ We then note the
following. \ 

\begin{lemma}
Suppose that
\[
\sigma_{k}(D^{2}u)>0\text{ }%
\]
and $D^{2}u$ has at most $1$ negative eigenvalue. \ Then $D^{2}u\in\Gamma
_{k}^{+}.$
\end{lemma}

\begin{proof}
Diagonalize  $D^{2}u$ so that $D^{2}u=diag\left\{  \lambda_{1},\lambda
_{2},...,\lambda_{n}\right\}  $ with $0\leq\lambda_{2}\leq\lambda_{3}%
...\leq\lambda_{n}.$ \ Clearly
\[
\frac{d}{ds}\sigma_{k}(diag\left\{  \lambda_{1}+s,\lambda_{2},...,\lambda
_{n}\right\}  )\geq0
\]
so we may deform $D^{2}u$ to a positive definite matrix $D^{2}u+M$, with
$\sigma_{k}(D^{2}u+sM)>0$ for $s\geq0.$ Thus $D^{2}u$ is in the component of
$\sigma_{k}>0$ containing the positive cone, that is, $D^{2}u\in\Gamma_{k}%
^{+}.$
\end{proof}

\begin{example}
When $n=3$ the function
\[
u(x,y,t)=(x^{2}+y^{2})e^{t}+\frac{1}{4}e^{-t}-e^{t}%
\]
solves
\[
\sigma_{2}(D^{2}u)=1.
\]

\end{example}

\begin{remark}
This method allows one to construct solutions to complex Monge-Amp\`{e}re
equations as well. \ See \cite{WarrenDonaldson}. \ 
\end{remark}

\bibliographystyle{plain}
\bibliography{sigma_k}

\begin{thebibliography}{1}

\bibitem{MR0106487}
Eugenio Calabi.
\newblock Improper affine hyperspheres of convex type and a generalization of a
  theorem by {K}. {J}\"orgens.
\newblock {\em Michigan Math. J.}, 5:105--126, 1958.
\newblock folder 5.

\bibitem{CY2010}
Sun-Yung~Alice Chang and Yu~Yuan.
\newblock A {L}iouville problem for the sigma-2 equation.
\newblock {\em Discrete Contin. Dyn. Syst.}, 28(2):659--664, 2010.

\bibitem{Fu}
Lei Fu.
\newblock An analogue of {B}ernstein's theorem.
\newblock {\em Houston J. Math.}, 24(3):415--419, 1998.

\bibitem{HL}
Reese Harvey and H.~Blaine Lawson, Jr.
\newblock Calibrated geometries.
\newblock {\em Acta Math.}, 148:47--157, 1982.

\bibitem{J54}
Konrad J{\"o}rgens.
\newblock \"{U}ber die {L}\"osungen der {D}ifferentialgleichung {$rt-s^2=1$}.
\newblock {\em Math. Ann.}, 127:130--134, 1954.

\bibitem{Pog72}
A.~V. Pogorelov.
\newblock On the improper convex affine hyperspheres.
\newblock {\em Geometriae Dedicata}, 1(1):33--46, 1972.

\bibitem{WarrenDonaldson}
Micah Warren.
\newblock A {B}ernstein result and counterexample for entire solutions to
  {D}onaldson's equation.
\newblock {\em arXiv:1503.06847}.

\bibitem{YuanInventiones}
Yu~Yuan.
\newblock A {B}ernstein problem for special {L}agrangian equations.
\newblock {\em Invent. Math.}, 150(1):117--125, 2002.

\bibitem{YY06}
Yu~Yuan.
\newblock Global solutions to special {L}agrangian equations.
\newblock {\em Proc. Amer. Math. Soc.}, 134(5):1355--1358 (electronic), 2006.

\end{thebibliography}

\end{document}